\documentclass[a4paper,12pt,final]{amsart}
\usepackage{times,a4wide,mathrsfs,amssymb}

\newcommand{\C}{\mathbb{C}}
\newcommand{\Z}{\mathbb{Z}}

\newcommand{\QQ}{\mathbb{Q}}
\newcommand{\NN}{\mathbb{N}}
\newcommand{\PP}{\mathbb{P}}

\newcommand{\OO}{\mathcal O}
\newcommand{\Ss}{\mathcal S}

\newcommand{\MM}{\mathcal M}

\newcommand{\hh}{\mathfrak h}
\newcommand{\ttt}{\mathfrak t}

\newcommand{\wt}{\widetilde}
\newcommand{\ima}{\hbox{Im}}
\newcommand{\rom}{\romannumeral}
\newcommand{\alb}{\hbox{Alb}}
\newcommand{\ide}{\hbox{id}}
\newcommand{\aut}{\hbox{Aut}}

\newtheorem{theorem}{Theorem}[section]

\newtheorem{lemma}[theorem]{Lemma}
\newtheorem{corollary}[theorem]{Corollary}
\newtheorem{proposition}[theorem]{Proposition}
\newtheorem{conjecture}[theorem]{Conjecture}
\newtheorem{nonumbering}{Theorem}

\newtheorem{convention}{Conventions}

\theoremstyle{definition}
\newtheorem{remark}[theorem]{Remark}
\newtheorem{definition}[theorem]{Definition}

\newtheorem{nonumberingt}{Acknowledgements}

\begin{document}
\author[Robert Laterveer]
{Robert Laterveer}

\address{Institut de Recherche Math\'ematique Avanc\'ee,
CNRS -- Universit\'e 
de Strasbourg,\
7 Rue Ren\'e Des\-car\-tes, 67084 Strasbourg Cedex,
FRANCE.}
\email{robert.laterveer@math.unistra.fr}

\title{A remark on the Chow ring of Sicilian surfaces}

\begin{abstract} We propose a ``Bloch type'' conjecture for surfaces: if the cup product map in coherent cohomology is zero, then all intersections of homologically trivial divisors should be zero in the Chow group of zero--cycles. We prove this conjecture for Sicilian surfaces. 
\end{abstract}



\maketitle

\section{Introduction}

For $X$ a smooth projective variety over $\C$, let $A^j(X)$ denote the Chow groups of codimension $j$ algebraic cycles on $X$ modulo rational equivalence. The intersection product makes a graded ring of $A^\ast(X)=\oplus_j A^j(X)$, the {\em Chow ring\/} of $X$.

In this note, we will be interested in the Chow ring of smooth projective surfaces $S$. What can be said about the image of the intersection product map
  \[   i_S\colon\ \ A^1(S)\otimes A^1(S)\ \to\ A^2(S)\ \ \ ?\]
  For $K3$ surfaces, the image of $i_S$ is as small as possible: it is a free abelian group of rank $1$ \cite{BV}.  
  At the other extreme, for abelian surfaces the map $i_S$ is surjective (the same is true for the Fano surface of lines on a cubic threefold \cite{B}, and another example where this holds is given in remark \ref{compare} below).
  For surfaces $S\subset\PP^3$, the rank of the image of $i_S$ 
  can grow arbitrarily large \cite{OG}.
  
 There is a relation with the cohomology ring: if $i_S$ is surjective, then also the cup product map in coherent cohomology
   \[  H^1(S,\OO_S)\otimes  H^1(S,\OO_S)\ \to\ H^2(S,\OO_S)\]
   is surjective \cite{ESV}. The conjectural converse statement is studied in \cite{moimult}. To complete the picture, we propose the following conjecture:
   
  \begin{conjecture}\label{conj} Let $S$ be a smooth projective surface, such that the cup product map
    \[  H^1(S,\OO_S)\otimes  H^1(S,\OO_S)\ \to\ H^2(S,\OO_S)\]
is zero. Then the intersection product map
  \[ j_S\colon\ \ \ A^1_{hom}(S)\otimes A^1_{hom}(S)\ \to\ A^2_{AJ}(S) \]
  is also zero.  
  \end{conjecture}
  
  Here, $A^1_{hom}$ denotes homologically trivial cycles, and $A^2_{AJ}$ denotes the Albanese kernel. The point of conjecture \ref{conj} is that 
  $A^2_{AJ}(S)$ is expected to be related to $H^2(S,\OO_S)$ \cite{B}.  
  A particular case of conjecture \ref{conj} is that the map $j_S$ should be zero for any surface with irregularity $q(S):=h^{1,0}(S)=1$.
  
  We can prove conjecture \ref{conj} for so--called {\em Sicilian surfaces\/}. These surfaces (defined in \cite{BCF}, cf. also definition \ref{sic} below) form a $4$--dimensional family of general type surfaces with $p_g(S):=h^{2,0}(S)=q(S)=1$ and $K_S^2=6$.
    
  \begin{nonumbering}[=theorem \ref{main}] Let $S$ be a Sicilian surface as in \cite{BCF}. Then the intersection product map
  \[ j_S\colon\ \ \ A^1_{hom}(S)\otimes A^1_{hom}(S)\ \to\ A^2_{AJ}(S) \]
  is zero.    
  \end{nonumbering}
  
  This implies that the image of $i_S$ is ``not so large'' for Sicilian surfaces: it is supported on a divisor (corollary \ref{cor}). The proof of theorem \ref{main} is an easy application of O'Sullivan's theory of {\em symmetrically distinguished cycles\/} on abelian varieties \cite{OS} (cf. also subsection \ref{ssd} below).
  

\vskip0.6cm

\begin{convention} In this note, the word {\sl variety\/} will refer to a reduced irreducible scheme of finite type over $\C$, and a {\sl surface\/} will mean a $2$--dimensional variety. For any variety $X$, we will denote by $A_j(X)$ the Chow group of dimension $j$ cycles on $X$. For a smooth $n$--dimensional variety $X$, we will write
$A^j(X)$ for $A_{n-j}(X)$. For a smooth proper variety, $A^j_{hom}(X)$ and $A^j_{AJ}(X)$ will be used to indicate the subgroups of 
homologically trivial, resp. Abel--Jacobi trivial cycles.
For a morphism between smooth varieties $f\colon X\to Y$, we will write $\Gamma_f\in A^\ast(X\times Y)$ for the graph of $f$, and ${}^t \Gamma_f\in A^\ast(Y\times X)$ for the transpose.

The contravariant category of Chow motives (i.e., pure motives with respect to rational equivalence as in \cite{Sc}, \cite{MNP}) will be denoted $\MM_{\rm rat}$. 

We will write $H^j(X)$ to indicate singular cohomology $H^j(X,\QQ)$.
\end{convention}

\section{Preliminaries}

\subsection{Chow cohomology}

For a singular variety $X$, we follow the convention of \cite{F} and write $A_\ast(X)$ for Chow groups and $A^\ast(X)$ for the {\em operational Chow cohomology} of \cite[Chapter 17]{F}. As proven in loc. cit., $A^\ast()$ is a contravariant functor from varieties to commutative rings, and for $X$ smooth the ring structure coincides with the usual intersection product. For $n$--dimensional quotient varieties $X=Y/G$ with $Y$ smooth and $G$ finite, the natural map induces isomorphisms
  \[ A^i(X)_{\QQ}\ \xrightarrow{\cong}\ A_{n-i}(X)_{\QQ}\ \ \ \forall i\ \]
  \cite[Example 17.4.10]{F}. The same is true for surfaces whose singularities are rational \cite[Theorem 4.1]{Vi}, \cite{Kim0}.
%
%
%

\subsection{Finite--dimensional motives}

We refer to \cite{Kim}, \cite{An}, \cite{J4}, \cite{MNP} for the definition of finite--dimensional motive. 
An essential property of varieties with finite--dimensional motive is embodied by the nilpotence theorem:

\begin{theorem}[Kimura \cite{Kim}]\label{nilp} Let $X$ be a smooth projective variety of dimension $n$ with finite--dimensional motive. Let $\Gamma\in A^n(X\times X)_{\QQ}$ be a correspondence which is numerically trivial. Then there is $N\in\NN$ such that
     \[ \Gamma^{\circ N}=0\ \ \ \ \in A^n(X\times X)_{\QQ}\ .\]
\end{theorem}

 Actually, the nilpotence property (for all powers of $X$) could serve as an alternative definition of finite--dimensional motive, as shown by a result of Jannsen \cite[Corollary 3.9]{J4}.
Conjecturally, any variety has finite--dimensional motive \cite{Kim}. We are still far from knowing this, but at least there are quite a few non--trivial examples.

 \subsection{The transcendental motive}

\begin{theorem}[Kahn--Murre--Pedrini \cite{KMP}]\label{t2} Let $S$ be a surface. There exists a decomposition
  \[ \hh^2(S)= \ttt^2(S)\oplus \hh^2_{alg}(S)\ \ \ \hbox{in}\ \MM_{\rm rat}\ ,\]
  such that
  \[  H^\ast(\ttt^2(S),\QQ)= H^2_{tr}(S)\ ,\ \ H^\ast(\hh^2_{alg}(S),\QQ)=NS(S)_{\QQ}\ \]
  (here $H^2_{tr}(S)$ is defined as the orthogonal complement of the N\'eron--Severi group $NS(S)_{\QQ}$ in $H^2(S,\QQ)$),
  and
   \[ A^\ast(\ttt^2(S))_{\QQ}=A^2_{AJ}(S)_{\QQ}\ .\]
   (The motive $\ttt^2(S)$ is called the {\em transcendental part of the motive\/}.)
   \end{theorem} 
   
  \begin{remark} It would be more precise to write $H^\ast(\hh^2_{alg}(S),\QQ)=NS(S)_{\QQ}(-1)$, taking into account the Tate twist. In this note, we will omit Tate twists from the notation.  
  \end{remark}

 \subsection{Refined Chow--K\"unneth decomposition}
 
 \begin{theorem}[Vial \cite{V4}]\label{pi_2} Let $X$ be a smooth projective variety of dimension $n\le 5$. Assume that $X$ has finite--dimensional motive, and that the Lefschetz standard conjecture $B(X)$ holds (in particular, the K\"unneth components $\pi_i\in H^{2n}(X\times X)$ are algebraic). Then there is a splitting into mutually orthogonal idempotents
  \[ \pi_i=\sum_j \pi_{i,j}\ \ \ \in A^{n}(X\times X)_{\QQ}\ ,\]
  such that
   \[ (\pi_{i,j})_\ast H^\ast(X) =gr^j_{\wt{N}} H^i(X)\  .\]
 (Here, $  gr^j_{\wt{N}} H^i(X)$ denotes the graded quotient for the {\em niveau filtration\/} $\wt{N}^\ast$ defined in \cite{V4}.)  

 The motive $\hh^{i,0}(X)=(X,\pi_{i,0},0)\in \MM_{\rm rat}$ is well--defined up to isomorphism. 

 \end{theorem}
 
 \begin{proof} This is \cite[Theorems 1 and 2]{V4}. The last statement follows from \cite[Proposition 1.8]{V4} combined with \cite[Theorem 7.7.3]{KMP}.
 \end{proof}
 
 \begin{remark} In dimension $n\le 3$ (which will be the case when we apply theorem \ref{pi_2} in this note), the niveau filtration $\wt{N}^\ast$ coincides with the coniveau filtration $N^\ast$ of \cite{BO}.
 
 \end{remark}

 \begin{remark} In dimension $n=2$, the motive $\hh^{2,0}(X)$ is isomorphic to the motive $\ttt^2(X)$ of theorem \ref{t2}.
 
 \end{remark}

\subsection{Symmetrically distinguished cycles}
\label{ssd}

\begin{definition}[O'Sullivan \cite{OS}] Let $A$ be an abelian variety. Let $a\in A^\ast(A)$ be a cycle. For $m\ge 0$, let
  \[ V_m(a)\ \subset\ A^\ast(A^m)_{\QQ} \]
  denote the $\QQ$--vector space generated by elements
  \[ p_\ast \Bigl(  (p_1)^\ast(a^{r_1})\cdot (p_2)^\ast(a^{r_2})\cdot\ldots\cdot (p_n)^\ast(a^{r_n})\Bigr)\ \ \ \in A^\ast(A^m)_{\QQ} \ . \]
Here $n\le m$, and $r_j\in\NN$, and $p_i\colon A^n\to A$ denotes projection on the $i$--th factor, and $p\colon A^n\to A^m$ is a closed immersion with each component $A^n\to A$ being either a projection
or the composite of a projection with $[-1]\colon A\to A$.

The cycle $a\in A^\ast(A)_{\QQ}$ is said to be {\em symmetrically distinguished\/} if for every $m\in\NN$ the composition
  \[ V_m(a)\ \subset\ A^\ast(A^m)_{\QQ}\ \to\ A^\ast(A^m)_{\QQ}/A^\ast_{hom}(A^m)_{\QQ} \]
  is injective.
\end{definition}

\begin{theorem}[O'Sullivan \cite{OS}]\label{os} The symmetrically distinguished cycles form a $\QQ$--subalgebra $A^\ast_{sym}(A)_{\QQ}\subset A^\ast(A)_{\QQ}$, and the composition
  \[  A^\ast_{sym}(A)_{\QQ}\ \subset\ A^\ast(A)_{\QQ}\ \to\ A^\ast(A)_{\QQ}/A^\ast_{hom}(A)_{\QQ} \]
  is an isomorphism. Symmetrically distinguished cycles are stable under pushforward and pullback of homomorphisms of abelian varieties.
\end{theorem}

\begin{remark} For discussion and applications of the theory of symmetrically distinguished cycles, in addition to \cite{OS} we refer to \cite[Section 7]{SV}, \cite{V6}, \cite{Anc}, \cite{LFu2}, \cite{FV}.
\end{remark}

\begin{proposition}\label{sym} Let $A$ be an abelian variety of dimension $g$. 

\noindent
(\rom1)
There exists a Chow--K\"unneth decomposition $\{ \Pi^i_A\}$ that is self--dual and consists of symmetrically distinguished cycles. One has equality
  \[  (\Pi_A^{2i-j})_\ast A^i(A)_{\QQ}=  A^i_{(j)}(A)_{}\ \ \ \forall i,j\ ,\]
 where $A^\ast_{(\ast)}(A)_{}$ denotes Beauville's decomposition \cite{Beau} on Chow groups with rational coefficients.

\noindent(\rom2) Assume $g\le 5$, and let  $\{ \Pi^i_A\}$ be as in (\rom1). There exists a further splitting in orthogonal projectors
  \[ \Pi^2_A= \Pi^{2,0}_A +\Pi^{2,1}_A\ \ \ \hbox{in}\ A^g(A\times A)_{\QQ}\ ,\]
  where the $\Pi^{2,i}_A$ are symmetrically distinguished and $\Pi_A^{2,i}=\pi_A^{2,i}$ in $H^{2g}(A\times A)$. Moreover, one has
    \[ (\Pi_A^{2,0})_\ast A^2(A)_{\QQ} =  (\Pi_A^{2})_\ast A^2(A)_{\QQ} = A^2_{(2)}(A)_{} \ .\]
   
\end{proposition}

 \begin{proof}

 \noindent
 (\rom1) An explicit formula for  $\{ \Pi^i_A\}$ is given in \cite[Section 7 Formula (45)]{SV}.

 \noindent
 (\rom2) The point is that $\Pi^{2,1}_A$ is (by construction) a cycle of type
   \[  \sum_j C_j\times D_j\ \ \ \hbox{in}\ A^g(A\times A)_{\QQ}\  ,\]
   where $D_j\subset A$ is a symmetric divisor and $C_j\subset A$ is a curve obtained by intersecting a symmetric divisor with hyperplanes. This implies $\Pi_A^{2,1}$ is symmetrically distinguished.
   By assumption, $\Pi^A_2$ is symmetrically distinguished and hence so is $\Pi^A_{2,0}$. 
   
  For the ``moreover'' part, one notes that the projector $\Pi_A^{2,1}$ acts trivially on $A^2_{(2)}(A)_{} \subset A^2_{AJ}(A)_{\QQ}$, for reasons of dimension.
  
 \end{proof}

  \subsection{Sicilian surfaces}
  
 \begin{definition}\label{sic}
  A {\em Sicilian surface\/} is a minimal surface $S$ of general type satisfying:

\noindent{(1)} $p_g(S)=q(S)=1$ and $K_S^2=6$;

\noindent{(2)} There exists an unramified double cover $\hat{S}\to S$ with $q(\hat{S})=3$, and such that the Albanese morphism $\hat{\alpha}\colon \hat{S}\to A:=\alb(\hat{S})$ is birational to its image $Z$, a divisor in $A$ with $Z^3=12$.
\end{definition}

\begin{remark} Sicilian surfaces have an irreducible $4$--dimensional moduli space \cite[Theorem 6.1]{BCF}. Sicilian surfaces can be caracterized topologically; they form a connected component of the moduli space of surfaces of general type \cite[Corollary 6.5]{BCF}. Surfaces in the families $\Ss_{11}$ and $\Ss_{12}$ constructed in \cite{BCF} are Sicilian surfaces.
\end{remark}
  
  We mention in passing the following result, which will {\em not\/} be used in the proof of the main result (theorem \ref{main}). 
  
  \begin{theorem}[Peters \cite{Pet}]\label{chris} Let $S$ be a Sicilian surface. Then $S$ has finite--dimensional motive. More precisely, let $A$ be the abelian threefold as in definition 
   \ref{sic}. Then the natural map
    \[ \hh^{2,0}(A)\ \to\ \ttt^2(S) \ \ \ \hbox{in}\ \MM_{\rm rat}\]
   admits a right--inverse, and the natural map
    \[ \ttt^2(S)\ \to\ \hh^4(A)  \ \ \ \hbox{in}\ \MM_{\rm rat}\]
    admits a left--inverse.
       \end{theorem}

%

  \section{Main result}

\begin{theorem}\label{main} Let $S$ be a Sicilian surface. The map induced by intersection product
  \[ j_S\colon\ \ \ A^1_{hom}(S)\otimes A^1_{hom}(S)\ \to\ A^2(S) \]
  is the zero map.
  \end{theorem}
  
 \begin{proof} As the image of $j_S$ is contained in $A^2_{AJ}(S)$ which is torsion free \cite{Ro}, 
 it will suffice to prove that $j_S\otimes\QQ$ is the zero map.
 
 The next reduction step is to pass to the canonical model $S_{can}$. Let $f\colon S\to S_{can}$ the canonical morphism. There is a commutative diagram
   \[ \begin{array}[c]{ccc}
       A^1_{hom}(S)\otimes A^1_{hom}(S) & \xrightarrow{j_S} & A^2(S)\\
       &&\\
      \ \ \ \ \  \uparrow {\scriptstyle (f^\ast,f^\ast)}&&\ \ \ \uparrow {\scriptstyle f^\ast}\\
      &&\\
      A^1_{hom}(S_{can})\otimes A^1_{hom}(S_{can}) & \xrightarrow{j_{S_{can}}} & A^2(S_{can})\\ 
      \end{array} \]
      where $A^2(S_{can})$ denotes operational Chow cohomology.
     The vertical arrows are isomorphisms (for the left vertical arrow, this is because $S_{can}$ has rational singularities, for the right vertical arrow this follows from the exact sequence of \cite[Theorem 2.3]{Kim0}). It thus suffices to prove that $j_{S_{can}}\otimes\QQ$ is the zero map.

%
 As shown in \cite[Theorem 6.1]{BCF}, the surface $S_{can}$ admits an inclusion as an ample divisor 
  \[ S_{can} \ \subset\ X=A/G\ ,\]
 where $X$ is a Bagnera--de Franchis threefold (in the sense of \cite[Section 5]{BCF}), and $A$ is the abelian threefold of definition \ref{sic} and $G\cong\Z_2$.
 Because $q(X)=q(S)=1$, the cup product map
   \[ H^{1}(X,\OO_X)\otimes H^1(X,\OO_X)\ \to\ H^2(X,\OO_X) \]
   is the zero map.
   In view of the Hodge decomposition, this means that the composition
   \[ H^1(X,\C)\otimes H^1(X,\C)\ \to\ H^2(X,\C)\ \to\ H^2(X,\OO_X)\ , \]
   which is the same as
   \[ H^1(A,\C)^G\otimes H^1(A,\C)^G\ \to\ H^2(A,\C)^G\ \to\ H^2(A,\OO_A)^G\ , \]
    is the zero map. 
  
  In terms of motives, this means that the composition
    \[ \hh^1(A)^G\otimes \hh^1(A)^G\ \xrightarrow{\Delta_A^{sm}}\ \hh^2(A)\ \xrightarrow{\pi_A^{2,0}}\ \hh^{2,0}(A) \ \ \ \hbox{in}\ \MM_{\rm hom}\]
    is zero (where $\Delta_A^{sm}\in A^6(A\times A\times A)$ is the ``small diagonal'', and the motive
    $\hh^{2,0}(A)\subset \hh^2(A)$ is as in proposition \ref{sym}). In terms of correspondences, this means that
  the correspondence
   \[   \Gamma:=  \Pi_A^{2,0}\circ \Delta_A^{sm}  \circ  \Bigl(\Pi^1_A\circ (\Delta_A+\Gamma_g)\bigr)\times  \bigl(\Pi^1_A\circ (\Delta_A+\Gamma_g)\Bigr)\ \ \in\ A^{6}\bigl( (A\times A)\times A\bigr)_{\QQ} \]
   is homologically trivial (i.e., it vanishes in $H^{12}(A\times A\times A,\C)$ and hence also in $H^{12}(A\times A\times A,\QQ)$. Here, we have written $G=\{\ide,g\}\cong\Z_2$ (i.e., $g$ is the non--trivial element of $G$), and $\Pi^1_A, \Pi_A^{2,0}$ are the projectors of proposition \ref{sym}.
   
  The involution $g\in\aut(A)$ is described explicitly in \cite[Theorem 6.1]{BCF}; it can be written as a group homomorphism $\sigma$ followed by a translation $t$ by a torsion element. In view of lemma \ref{same} below, the graphs $\Gamma_g$ and $\Gamma_\sigma$ are the same in the Chow group with rational coefficients. Therefore, we have 
 equality
   \[ \Gamma= \Pi_A^{2,0}\circ \Delta_A^{sm}  \circ  \Bigl(\Pi^1_A\circ (\Delta_A+\Gamma_\sigma)\bigr)\times  \bigl(\Pi^1_A\circ (\Delta_A+\Gamma_\sigma)\Bigr)\ \ 
       \hbox{in}\      A^{6}\bigl( (A\times A)\times A\bigr){}_{\QQ}\ . \]   
   But the right--hand side (being a composition of symmetrically distinguished cycles) is symmetrically distinguished. Therefore, theorem \ref{os} implies that
   \[ \Gamma=0\ \ \ \hbox{in}\ A^{6}\bigl( (A\times A)\times A\bigr){}_{\QQ}\ . \]   
In particular, the action on Chow groups 
  \[ \Gamma_\ast\colon\ \ \ A^2(A\times A)_{\QQ}\ \to\ A^2(A)_{\QQ} \]
  is zero. On the other hand, let $a,b\in A^1_{hom}(A)^G$ and consider the element
  \[  a\times b\ \ \in\ \ima\Bigl(A^1_{hom}(A)^G\otimes A^1_{hom}(A)^G\ \to\ A^2(A\times A)\Bigr) \ .\]
 Then (by construction of $\Gamma$) we have equality
  \[   \Gamma_\ast(a\times b)= 4 \,(\Pi_A^{2,0})_\ast (\Delta_A^{sm})_\ast (a\times b) = 4\, (\Pi_A^{2,0})_\ast (a\cdot b)= 4 \,a\cdot b   \ \ \ \hbox{in}\ A^2(A)_{\QQ}\ .\]
  (Here, for the last equality we have used that $a\cdot b\in A^2_{(2)}(A)_{}$, as the Beauville decomposition of $A^\ast(A)$ is multiplicative.)
  The commutative diagram
  \[ \begin{array}[c]{ccc}
      A^1_{hom}(A)^G\otimes A^1_{hom}(A)^G &\xrightarrow{j_A} & A^2_{AJ}(A)^G\\
      &&\\
       \downarrow{\cong} &&\downarrow{\cong}\\
      &&\\
      A^1_{hom}(X)\otimes A^1_{hom}(X) &\xrightarrow{j_X} & A^2_{AJ}(X)\\            &&\\
     \ \ \ \  \downarrow{\scriptstyle (\iota^\ast,\iota^\ast)} &&\ \ \ \downarrow{\scriptstyle \iota^\ast}\\
      &&\\
       A^1_{hom}(S_{can})\otimes A^1_{hom}(S_{can}) &\xrightarrow{j_{S_{can}}} & \ A^2_{AJ}(S_{can})\ ,\\ 
        \end{array}\]
   plus the fact that $\iota^\ast\colon A^1_{hom}(X)\to A^1_{hom}(S_{can})$ is an isomorphism (weak Lefschetz), now ends the proof.
   
   In the above argument we have used the following, which is \cite[Lemma 2.1]{JY}:
          
   \begin{lemma}[\cite{JY}]\label{same} Let $A$ be an abelian variety of dimension $g$, and let $t\in\aut(A)$ be a translation by a torsion element. Then
     \[ \Gamma_t =\Delta_A\ \ \ \hbox{in}\ A^g(A\times A)_{\QQ}\ .\]
    \end{lemma}

 \end{proof}

  \begin{corollary}\label{cor} Let $S$ be a Sicilian surface. The image of the intersection product map
      \[   i_S\colon\ \ A^1(S)\otimes A^1(S)\ \to\ A^2(S) \]
      is supported on a divisor.  
   \end{corollary}
  
  \begin{proof} Let $D_1,\ldots,D_r$ be generators of the N\'eron--Severi group of $S$. Given arbitrary divisors $D, D^\prime\in A^1(S)$, let us write
  $D=\sum_{i=1}^r d_i D_i$, $D^\prime=\sum_{j=1}^r d^\prime_j D_j$ in $NS(S)$. This gives decompositions
    \[ D=\sum_{i=1}^r d_i D_i + D_0\ ,\ \ \ D^\prime= \sum_{j=1}^r d_j^\prime D_j + D_0^\prime\ \ \ \hbox{in}\ A^1(S)\ ,\]
    with $D_0, D_0^\prime\in A^1_{hom}(S)$.
  
  It follows from theorem \ref{main} that $D_0\cdot D_0^\prime=0$ in $A^2(S)$, and so
    \[ D\cdot D^\prime = \sum_{i=1}^r \sum_{j=1}^r d_i d^\prime_j D_i\cdot D_j  +  \sum_{i=1}^r d_i D_i\cdot D_0^\prime + \sum_{j=1}^r d^\prime_j D_j\cdot D_0   \ \ \ \hbox{in}\ A^2(S)\ .\]
  This implies that the image of $i_S$ is supported on the union $\cup_{j=1}^r D_r\subset S$.  
  \end{proof}

  \begin{remark}\label{compare} Theorem \ref{main} applies in particular to the generalized Burniat type surfaces in the families $\Ss_{11}$ and $\Ss_{12}$ of \cite{BCF} (as shown in loc. cit., these are Sicilian surfaces). It is instructive to contrast this with the behaviour of generalized Burniat type surfaces in the family $\Ss_{16}$ (these have $p_g(S)=q(S)=3$). Indeed, any surface $S$ in the family $\Ss_{16}$ has a surjective cup product map
    \[ H^{1}(S,\OO_S)\otimes H^{1}(S,\OO_S)\ \to\ H^2(S,\OO_S)\ .\]
    Moreover, $S$ has finite--dimensional motive \cite[Theorem 4.13]{BCF}. The main result of \cite{moimult} then implies that
    \[   i_S\colon\ \ A^1(S)\otimes A^1(S)\ \to\ A^2(S) \]
    is surjective (just as for abelian surfaces).    
    
%
  \end{remark}

 \begin{remark} The argument of theorem \ref{main} applies in a more general setting: it suffices that $S$ be a surface with $q(S)=1$ obtained as the resolution of a nodal surface $S_{can}$, which can be embedded as ample divisor 
 \[  S_{can}\ \ \subset\ X=A/G\ ,\] 
 where $A$ is an abelian threefold, and $G$ a finite group acting by compositions of translations and group homomorphisms. It follows that theorem \ref{main} is also true for the generalized Burniat type surfaces in the families $\Ss_j$, $5\le j\le 12$ of \cite{BCF}.
 \end{remark} 
 
 \begin{remark} As Sicilian surfaces (and generalized Burniat type surfaces) are closely related to abelian varieties, it seems natural to ask whether they admit a {\em multiplicative Chow--K\"unneth decomposition\/}, in the sense of \cite[Section 8]{SV}. I hope to return to this question later.
 \end{remark}

\vskip1cm
\begin{nonumberingt} Thanks to Kai and Len, my dear colleagues at the Schiltigheim Math Research Institute. Thanks to Chris Peters for helpful discussions concerning \cite{Pet}, and to the referee for insightful comments that helped to improve this paper.
\end{nonumberingt}


\vskip1cm
\begin{thebibliography}{dlPG99}


\bibitem{Anc} G. Ancona, D\'ecomposition de motifs ab\'eliens, Manuscripta Math. 146 (3) (2015), 307---328,

\bibitem{An} Y. Andr\'e, Motifs de dimension finie (d'apr\`es S.-I. Kimura, P. O'Sullivan,...), S\'eminaire Bourbaki 2003/2004, Ast\'erisque 299 Exp. No. 929, viii, 115---145,

\bibitem{BCF} I. Bauer, F. Catanese and D. Frapporti, Generalized Burniat type surfaces and Bagnera--de Franchis varieties, J. Math. Sci. Univ. Tokyo 22 (2015), 
55---111,

\bibitem{Beau} A. Beauville, Sur l'anneau de Chow d'une vari\'et\'e ab\'elienne, Math. Ann. 273 (1986), 647---651,


\bibitem{BV} A. Beauville and C. Voisin, On the Chow ring of a $K3$ surface, J. Alg. Geom. 13 (2004), 417---426,


\bibitem{B} S. Bloch, Lectures on algebraic cycles, Duke Univ. Press Durham 1980,

\bibitem{BO} S. Bloch and A. Ogus, Gersten's conjecture and the homology of schemes, Ann. Sci. Ecole Norm. Sup. 4 (1974), 181---202,



\bibitem{ESV} H. Esnault, V. Srinivas and E. Viehweg, Decomposability of Chow groups implies decomposability of cohomology, in: ``Journ\'ees de G\'eom\'etrie Alg\'ebrique d'Orsay, Juillet 1992'', Ast\'erisque 218 (1993), 227---242, 


\bibitem{F} W. Fulton, Intersection theory, Springer--Verlag Ergebnisse der Mathematik, Berlin Heidelberg New York Tokyo 1984,

\bibitem{LFu2} L. Fu, Beauville-Voisin conjecture for generalized Kummer varieties, International Mathematics Research Notices 12 (2015), 3878---3898,

\bibitem{FV} L. Fu and Ch. Vial, Distinguished cycles on varieties with motives of abelian type and the Section Property, arXiv:1709.05644v1,

\bibitem{J4} U. Jannsen, On finite--dimensional motives and Murre's conjecture, in: Algebraic cycles and motives (J. Nagel and C. Peters, eds.), Cambridge University Press, Cambridge 2007,

\bibitem{JY} Z. Jiang and Q. Yin, On the Chow ring of certain rational cohomology tori, Comptes rendus Acad. Sci. Paris 355 (2017), 571---576, 

\bibitem{KMP} B. Kahn, J. Murre and C. Pedrini, On the transcendental part of the motive of a surface, in: Algebraic cycles and motives (J. Nagel and C. Peters, eds.), Cambridge University Press, Cambridge 2007,

\bibitem{Kim0} S. Kimura, Fractional intersection and bivariant theory, Communications in Algebra 20(1) (1992), 285---302,

\bibitem{Kim} S. Kimura, Chow groups are finite dimensional, in some sense,
Math. Ann. 331 (2005), 173---201,



\bibitem{moimult} R. Laterveer, On a multiplicative version of Bloch's conjecture, Beitr\"age zum Algebra und Geometrie 57(4) (2016), 723---734,



\bibitem{MNP} J. Murre, J. Nagel and C. Peters, Lectures on the theory of pure motives, Amer. Math. Soc. University Lecture Series 61, Providence 2013,

\bibitem{OG} K. O'Grady, Decomposable cycles and Noether--Lefschetz loci, arXiv:1502.07897v2,

\bibitem{OS} P. O'Sullivan, Algebraic cycles on an abelian variety, J. f. Reine u. Angew. Math. 654 (2011), 1---81,

\bibitem{Pet} C. Peters, A motivic study of generalized Burniat surfaces, arXiv:1710.02370,

\bibitem{Ro} A.A. Rojtman, The torsion of the group of 0--cycles modulo rational equivalence, Annals of Mathematics 111 (1980), 553---569,


\bibitem{Sc} T. Scholl, Classical motives, in: Motives (U. Jannsen et alii, eds.), Proceedings of Symposia in Pure Mathematics Vol. 55 (1994), Part 1, 

\bibitem{SV} M. Shen and Ch. Vial, The Fourier transform for certain hyperK\"ahler fourfolds, Memoirs of the AMS 240 (2016), no.1139,


\bibitem{V4} Ch. Vial, Niveau and coniveau filtrations on cohomology groups and Chow groups, Proceedings of the LMS 106(2) (2013), 410---444,

\bibitem{V6} Ch. Vial, On the motive of some hyperk\"ahler varieties, J. f. Reine u. Angew. Math. 725 (2017), 235---247,

\bibitem{Vi} A. Vistoli, Alexander duality in intersection theory, Comp. Math. 70 no. 3 (1989), 199---225.



\end{thebibliography}
\end{document}